\documentclass[11pt,reqno,a4paper]{amsart}
\usepackage{amsfonts,latexsym,amssymb,amscd,amsmath,mathtools}
\usepackage{ifthen}
\usepackage{graphicx,bm}
\usepackage{dsfont}
\usepackage{color}

\nonstopmode \numberwithin{equation}{section}

\makeatletter
\newcommand*\bigcdot{\mathpalette\bigcdot@{.5}}
\newcommand*\bigcdot@[2]{\mathbin{\vcenter{\hbox{\scalebox{#2}{$\m@th#1\bullet$}}}}}
\makeatother

\newtheorem{thm}{Theorem}
\newtheorem{lem}{Lemma}
\newtheorem{cor}{Corollary}[section]

\newtheorem{cl}{Claim}
\newtheorem{ca}{Case}
\newtheorem{sca}{Case}
\newtheorem{scl}{Subclaim}
\newtheorem{conj}{Conjecture}

\theoremstyle{definition}
\newtheorem{defn}{Definition}

\newtheorem{op}[equation]{Open Problem}
\newtheorem{ques}[equation]{Question}
\newtheorem{rem}{Remark}[section]
\newtheorem{exam}[equation]{Example}

\newcounter {own}
\def\theown {\thesection       .\arabic{own}}

\newenvironment{pf}[1][]{%
	\vskip 3mm
	\noindent
	\ifthenelse{\equal{#1}{}}%
	{{\slshape Proof. }}%
	{{\slshape #1.} }%
}%
{\qed\bigskip}

\newcounter{alphabet}

\newenvironment{Thm}[1][]{\refstepcounter{alphabet}%
	\noindent%
	{\bf Theorem \Alph{alphabet}}%
	\ifthenelse{\equal{#1}{}}{}{ (#1)}%
	{\bf .} \itshape}{\vskip 8pt}


\newenvironment{Core}[1][]{\refstepcounter{alphabet}%
	\bigskip%
	\noindent%
	{\bf Corollary \Alph{alphabet}}%
	{\bf .} \itshape}{\vskip 8pt}

\newcommand{\B}{{\mathrm B}}
\newcommand{\BB}{{\mathbb B}}

\newcommand{\C}{{\mathcal C}}

\newcommand{\R}{{\mathbb R}}
\newcommand{\Sp}{{\mathbb S^{n-1}}}

\newcommand{\J}{{\mathcal J}}
\newcommand{\I}{{\mathcal I}}

\newcommand{\IN}{{\mathbb N}}

\newcommand{\K}{{\mathcal K}}

\newcommand{\Z}{{\mathbb Z}}

\def\be{\begin{equation}}
	\def\ee{\end{equation}}

\newcommand{\bee}{\begin{enumerate}}
	\newcommand{\eee}{\end{enumerate}}

\newcommand{\blem}{\begin{lem}}
	\newcommand{\elem}{\end{lem}}
\newcommand{\bthm}{\begin{thm}}
	\newcommand{\ethm}{\end{thm}}
\newcommand{\bcor}{\begin{cor}}
	\newcommand{\ecor}{\end{cor}}
\newcommand{\beg}{\begin{exam}}
	\newcommand{\eeg}{\end{exam}}
\newcommand{\begs}{\begin{examples}}
	\newcommand{\eegs}{\end{examples}}
\newcommand{\bdefe}{\begin{defn}}
	\newcommand{\edefe}{\end{defn}}
\newcommand{\bprob}{\begin{prob}}
	\newcommand{\eprob}{\end{prob}}
\newcommand{\bques}{\begin{ques}}
	\newcommand{\eques}{\end{ques}}
\newcommand{\bei}{\begin{itemize}}
	\newcommand{\eei}{\end{itemize}}
\newcommand{\bcon}{\begin{conj}}
	\newcommand{\econ}{\end{conj}}
\newcommand{\bop}{\begin{op}}
	\newcommand{\eop}{\end{op}}

\newcommand{\bca}{\begin{ca}}
	\newcommand{\eca}{\end{ca}}
\newcommand{\bsca}{\begin{sca}}
	\newcommand{\esca}{\end{sca}}

\newcommand{\bcl}{\begin{cl}}
	\newcommand{\ecl}{\end{cl}}

\newcommand{\bscl}{\begin{scl}}
	\newcommand{\escl}{\end{scl}}

\newcommand{\bcons}{\begin{conjs}}
	\newcommand{\econs}{\end{conjs}}
\newcommand{\bprop}{\begin{propo}}
	\newcommand{\eprop}{\end{propo}}
\newcommand{\br}{\begin{rem}}
	\newcommand{\er}{\end{rem}}
\newcommand{\brs}{\begin{rems}}
	\newcommand{\ers}{\end{rems}}
\newcommand{\bo}{\begin{obser}}
	\newcommand{\eo}{\end{obser}}
\newcommand{\bos}{\begin{obsers}}
	\newcommand{\eos}{\end{obsers}}
\newcommand{\bpf}{\begin{pf}}
	\newcommand{\epf}{\end{pf}}
\newcommand{\ba}{\begin{array}}
	\newcommand{\ea}{\end{array}}
\newcommand{\beq}{\begin{eqnarray}}
	\newcommand{\beqq}{\begin{eqnarray*}}
		\newcommand{\eeq}{\end{eqnarray}}
	\newcommand{\eeqq}{\end{eqnarray*}}

\everymath{\displaystyle}
	
\begin{document}
    
\title{Khavinson conjecture for hyperbolic harmonic functions on the unit ball}
\author{Adel Khalfallah}
\address{Department of Mathematics and Statistics, King Fahd University of Petroleum and
	Minerals, Dhahran 31261, Saudi Arabia}\email{khelifa@kfupm.edu.sa}
\author{ Fathi Haggui}
\address{Institut Pr\'eparatoire Aux Etude d’Ing\'enieurs de Monastir (IPEIM), Universit\'e  de Monastir, Tunisia}
\email{fathi.haggui@gmail.com}

\author{ Miodrag Mateljevi\'c}
\address{M. Mateljevi\'c, Faculty of mathematics, University of Belgrade, Studentski Trg 16, Belgrade, Republic of Serbia}
\email{miodrag@matf.bg.ac.rs}
\maketitle

\begin{abstract}
    In this paper, we prove the Khavinson conjecture for hyperbolic harmonic functions on the unit ball. This conjecture was partially solved in \cite{JKM2020}
\end{abstract}

\section{introduction}
For $n\geq 2$, let $\R^n$ denote the $n$-dimensional Euclidean space. We use 
$\BB^n$ and $\Sp$ to denote the unit ball and the unit sphere in $\R^n$, respectively.\\
A mapping $u \in \mathcal{C}^2(\BB^n,\R)$ is said to be hyperbolic harmonic if 
$\Delta_h u=0,$
where $\Delta_h$ is the hyperbolic Laplacian operator defined by
$$\Delta_h u(x)= (1-|x|^2)^2 \Delta u+ 2(n-2)(1-|x|^2) \sum_{i=1}^n x_i \frac{\partial u}{\partial x_i}(x),  $$
here $\Delta$ denotes the Laplacian on $\R^n$.

Clearly for $n=2$, hyperbolic harmonic and harmonic functions coincide.\\

If $\phi\in L^1(\Sp,\mathbb{R})$, we define the  invariant Poisson integral of $\phi$ in $\mathbb{B}^{n}$ 

$$
\mathcal{P}_h[\phi](x)=\int_{\mathbb{S}^{n-1}} \mathcal{P}_h(x,\zeta)\phi(\zeta)d\,\sigma(\zeta),
$$
where
\begin{eqnarray*}
\mathcal{P}_h(x,\zeta)=\left(\frac{1-|x|^2}{|x-\zeta|^{2}}\right)^{n-1}
\end{eqnarray*}
is the Poisson  kernel with respective to $\Delta_{h}$ satisfying
$$
\int_{\mathbb{S}^{n-1}} \mathcal{P}_h(x,\zeta)\,
d\sigma(\zeta)=1.
$$

\noindent For more information about  hyperbolic harmonic functions we refer to  Stoll
 \cite{sto2016} and Burgeth \cite{Burgeth,Burgeth1994}.

\section{Khavinson Problem}

Let $p\in (1,\infty]$ and let $q$ be its conjugate. Assume that $u=\mathcal{P}_h[\phi]$, where $\phi\in L^p(\Sp,\R)$. For $x\in \BB^n\setminus \{0\}$ and $\ell \in \Sp$, 
let $\C_p(x)$ and $\C_p(x;\ell)$ denote the optimal numbers such that 
$$|\nabla u(x)| \leq \C_p(x) \|\phi\|_p, $$
and 
$$| \langle\nabla u(x), \ell \rangle| \leq \C_p(x;\ell) \|\phi\|_p.$$
Since $\displaystyle |\nabla u(x)| = \sup_{\ell \in \Sp} |\langle \nabla u(x),\ell \rangle|$, clearly we obtain 
$$\C_p(x) =\sup_{\ell\ \in \Sp} \C_p(x;\ell).$$
We prove the Khavinson conjecture for hyperbolic harmonic functions, partially solved in \cite{JKM2020}.
\begin{conj} Let $p \in (1,\infty]$, $n \geq 3$ and $x \in \BB^n\setminus \{0\}$. Then
$$
 \C_p(x) = \begin{dcases*}
        \C_p(x;n_x)  & if $1<p< n$, \\
        \C_p(x;t_x) & if $p> n$,
        \end{dcases*}
$$
where $ n_x=\frac{x}{|x|}$, and  $t_x$ is any unit vector such that $\langle t_x, x\rangle =0$.\\
Moreover, if  $p=n$ or $p=\infty$, then $\C_p(x,\ell)$  does not depend on $\ell$. 
\end{conj}

In the planar case, i.e., $n=2$, this conjecture was solved by Kalaj and Markovi\'c,  see \cite[Theorem 1.1]{kama}.\\

 Khavinson \cite{khavinson}  obtained a sharp pointwise estimate for the radial derivative of bounded harmonic functions on the unit ball of 
$\R^3$ and  conjectured that the same estimate holds for the norm of the gradient of bounded harmonic functions.\\
For harmonic functions this conjecture was formulated by Kresin and Maz’ya in \cite{kr}  and in \cite{kr2} considered the half-space analogue of the above conjecture. See \cite[Chapter 6]{kr3} for 
various Khavinson-type extremal problems for harmonic functions.\\
 Kalaj \cite{K2017} showed that the conjecture  for $n=4$ and  Melentijević \cite{Mel} confirmed the conjecture in $\R^3$. Markovi\'c \cite{mark} solved the Khavinson problem for points near the boundary of the unit ball.  The general conjecture was recently proved by Liu \cite{liu}.\\

By computing the gradient of the Poisson-Szeg\"o kernel and using the M\"obius transformation as a substitution, we obtain the following integral representation  

\begin{lem}\label{lem-2.1}\cite{JKM2020}
For any $p\in (1,\infty]$, $x\in\mathbb{B}^{n}$ and $l\in \mathbb{S}^{n-1}$,
we have
\begin{equation}\label{eq-1.1}
     \C_p(x;\ell)=\frac{ 2(n-1)}{(1-|x|^2)^{\frac{n(q-1)+1}{q}}} \left(\int_{\mathbb{S}^{n-1}}  |\eta-x|^{2(n-1)(q-1)} |\langle\eta,l\rangle|^{q} \,  d\sigma(\eta)\right)^{\frac{1}{q}}.
\end{equation}
\end{lem}

Moreover, one can easily deduce the following 

\begin{lem}\cite{JKM2020}\label{lem-2.2}
For any $p\in(1,\infty]$, $x\in\mathbb{B}^{n}$, $l\in \mathbb{S}^{n-1}$ and unitary transformation $A$ in $\mathbb{R}^{n}$,
we have
\begin{equation}
\C_p(x;\ell)=\C_p(Ax;A\ell).
\end{equation}
\end{lem}

For $p \in (1,\infty]$ and $\ell \in \Sp$, let

\begin{equation}
    \K_p(x;\ell)= \int_{\mathbb{S}^{n-1}}  |\eta-x|^{2(n-1)(q-1)} |\langle\eta,l\rangle|^{q}\,   d\sigma(\eta).
\end{equation}

So in view of (\ref{eq-1.1}), we have 
\begin{equation}\label{CK}
 \C_p(x;\ell)=\frac{ 2(n-1)}{(1-|x|^2)^{\frac{n(q-1)+1}{q}}} \left( \K_p(x;\ell) \right)^{\frac{1}{q}}.
 \end{equation}
 
 Our main result is the  following theorem solving the Khavinson conjecture for hyperbolic harmonic functions. 
 
 \begin{thm}
 Let $p\in (1,\infty]$, $n \geq 3$ and $x \in \BB^n\setminus \{0\}$. Then  
 $$
 \max_{\ell \in \Sp} \C_p(x;\ell) = \begin{dcases*}
        \C_p(x;n_x)  & if $1<p< n$, \\
        \C_p(x;t_x) & if $p> n$.
        \end{dcases*}
$$

$$
\min_{\ell \in \Sp} \C_p(x;\ell) = \begin{dcases*}
        \C_p(x;t_x)  & if $1<p< n$, \\
        \C_p(x;n_x) & if $p> n$.
        \end{dcases*}
$$
If $p=n$ or $p=\infty$, then $\ell \mapsto \C_p(x;\ell)$ is constant.
 \end{thm}

One of our main tools is the method of slice integration on spheres.\\

\begin{Thm}\cite[Theorem A.5]{Axler1992}
Let $f$  be a Borel measurable, integrable function on $\Sp$. If $1\leq k<n$, then 
$$
\int_{\Sp} f d\sigma_n  = \frac{k}{n} \frac{V(\BB^k)}{V(\BB^n)}\int_{\BB^{n-k}} (1-|x|^2)^{\frac{k-2}{2}}\int_{\mathbb{S}^{k-1}} f(x,\sqrt{1-|x|^2}\zeta) d\sigma_k(\zeta)\,dV_{n-k}(x),  $$

where  $V(\BB^n)$ denotes the volume of the ball, which is given by
\begin{equation}
 V(\BB^n)= \frac{\pi^{\frac{n}{2}}}{\Gamma(\frac{n}{2}+1)}.
\end{equation}
and $\sigma_n$ denotes the normalized measure on the sphere $\Sp$.
\end{Thm}
We will consider two special cases for $k=n-1$ and $k=n-2$. The corresponding formulas are useful when the integrand function $f$ depends only on one or two variables.

\begin{Core}
Let $\eta=(\eta_1, \ldots,\eta_n)\in \Sp$ and  $f(\eta)$  be a Borel measurable, integrable function on $\Sp$. 
\begin{enumerate}
    \item If $n\geq 2$ and $f(\eta)$ depends only on the first  variable $\eta_1$, then
    \begin{equation}
\int_{\Sp} f(\eta_1) \,d\sigma_n(\eta)  = \frac{n-1}{n} \frac{V(\BB^{n-1})}{V(\BB^n)}\int_{-1}^{1} (1-t^2)^{\frac{n-3}{2}} f(t)\, dt.          
    \end{equation}

\item If $n \geq 3$ and $f(\eta)$ depends only on the first two variables $\eta_1,\eta_2$, then 
\begin{eqnarray}
\int_{\Sp} f(\eta_1,\eta_2) \,d\sigma_n(\eta) & =& \frac{n-2}{2\pi} \int_{\BB^2} (1-|z|^2)^{\frac{n-4}{2}}\, f(z)\, dA(z). \\
 &=& \frac{n-2}{2\pi} \int_{\BB^2} (1-r^2)^{\frac{n-4}{2}}\, f(r \cos \theta, r \sin \theta ) r \, drd\theta,\label{polar}
\end{eqnarray}
\end{enumerate}
where $dA(z)$ denotes the Lebesgue measure on the unit disc $\BB^2$. 

\end{Core}

Using the invariance of $\K_p$ by unitary transformations, see Lemma \ref{lem-2.2}, we may assume that:
  $$x=|x|e_1 \mbox { and } \ell=\ell_\gamma=\cos \gamma e_1+\sin \gamma e_2,
\mbox{ with } \gamma\in [0,\pi].$$

Let $\eta=(\eta_1,\ldots,\eta_n) \in \Sp$. Then
$$|\eta-x|^2=1+|x|^2-2|x|\eta_1, $$
and
$$\langle \eta,\ell_\gamma\rangle= \eta_1 \cos \gamma +\eta_2 \sin \gamma. $$\\
For $r,\rho \in (0,1)$, introduce the notation
\begin{equation}
    \J_q(r, \rho;\gamma)= \int_{-\pi}^{\pi} \left( 1+\rho^2-2\rho r \cos \theta \right)^{(n-1)(q-1)}   \left|   \cos(\theta- \gamma)\right|^q d\theta. 
\end{equation}

\begin{lem} Let $x\in \BB^n$, $1<p<\infty$ and $q$ its conjugate. Then
$$\K_p(x;\ell)=\K_p(|x|e_1;\ell_\gamma)=\frac{n-2}{2\pi} \int_0^1 (1-r^2)^{\frac{n-4}{2}}  r^{q+1}    \J_q(r,|x|;\gamma)\, dr.
$$
\end{lem}

\begin{proof} Using the invariance of $\K_p$ by unitary transformations, we may assume that
  $x=|x|e_1 \mbox { and } \ell=\ell_\gamma=\cos \gamma e_1+\sin \gamma e_2,
\mbox{ with } \gamma\in [0,\pi].$
\begin{equation}
    \K_p(|x|e_1;\ell_\gamma)= \int_{\Sp} \left( 1+|x|^2-2|x|\eta_1 \right)^{(n-1)(q-1)} \left|  \eta_1 \cos \gamma +\eta_2 \sin \gamma \right|^q \, d \sigma(\eta).
\end{equation}

\noindent As the integrand function depends only on $\eta_1$ and $\eta_2$, the method of slice integration on spheres reduces an integral on the sphere to some integral on the unit disc $\BB^2$.  Using polar coordinates on the unit disc, let us denote $\eta_1= r \cos \theta$ and $\eta_2=r \sin \theta$. Thus

$$\langle \eta,\ell_\gamma\rangle= \eta_1 \cos \gamma +\eta_2 \sin \gamma= r \cos(\theta-\gamma). $$
\end{proof}

To find the extreme values of  $\J_q(r,\rho;\gamma)$, we will consider the following  more general integral
\begin{equation}
    \I_{a,b}(\gamma)= \int_{-\pi}^{\pi} (A-B \cos \theta)^a \, |\cos  (\theta-\gamma)|^b\, d\theta.
\end{equation}

The function $\I_{a,b}$ has the following properties
\begin{enumerate}
    \item $\I_{a,b}$ is $\pi$- periodic. 
    \item $\I_{a,b}$ is an even function.
\end{enumerate}
Thus, we will consider the behaviour of $\I_{a,b}$ only on $[0, \frac{\pi}{2}]$ and we show that $\I_{a,b}$ is a monotonic on $[0,\frac{\pi}{2}]$,  thus the extreme values are reached at $\gamma=0$ and  $\gamma=\frac{\pi}{2}$. A special case was considered in \cite[Lemma 2.1]{kama}. 
\begin{lem}
Let $A,B$ be positive  numbers such that $0< B<A$, and $a$, $b$ are real numbers such that $b>0$. 
\begin{enumerate}
    \item If $a=0$ or $a=1$, then $\gamma \mapsto \I(a,b; \gamma)$ is constant.
    \item If $a \in (0,1)$, then $\gamma \mapsto \I(a,b; \gamma)$ is increasing on $[0,\pi/2]$. Thus $$\max_{\gamma \in [0, \pi/2]} \I(a,b;\gamma)=\I(a,b;\pi/2).$$
    \item If $a>1 $, then $\gamma \mapsto \I(a,b; \gamma)$ is decreasing on $[0,\pi/2]$. Thus $$\max_{\gamma \in [0, \pi/2]} \I(a,b;\gamma)=\I(a,b;0).$$
\end{enumerate}
\end{lem}

\begin{proof}
As the integrand function is $2\pi$ periodic with respect to $\theta$, we deduce that

$$ \I_{a,b}(\gamma)= \int_{-\pi}^{\pi} (A-B \cos (\theta+\gamma))^a \, |\cos  \theta|^b\, d\theta.$$
Therefore the mapping is differentiable and
$$\I_{a,b}'(\gamma)= aB \int_{-\pi}^{\pi}  \sin (\theta+\gamma)(A-B \cos (\theta+\gamma))^{a-1} \, |\cos  \theta|^b\, d\theta.
$$
Again, using the $2\pi$-periodicity of the integrand, we obtain

$$\I_{a,b}'(\gamma)= aB \int_{-\pi}^{\pi}  \sin \theta(A-B \cos \theta)^{a-1} \, |\cos  (\theta-\gamma)|^b\, d\theta.
$$
Next, we split the integral into two parts from $0$ to $\pi$ and from $\pi$ to $2\pi$. Using a substitution, we obtain

$$\I_{a,b}'(\gamma)= aB \int_{0}^{\pi}  \sin\theta \left[(A-B \cos \theta)^{a-1}-(A+B \cos \theta)^{a-1} \right] \, |\cos  (\theta-\gamma)|^b\, d\theta.$$

By considering the substitution $u=\theta -\pi/2$, we get  

$$\I_{a,b}'(\gamma)= a B \int_{-\pi/2}^{\pi/2}  \cos\theta \left[(A+B \sin \theta)^{a-1}-(A-B \sin \theta)^{a-1} \right] \, |\sin  (\theta-\gamma)|^b\, d\theta.$$
Next, we split the integral into two parts from $0$ to $\pi/2$ and from $\pi/2$ to $\pi$. Using the substitution $u=\pi-\theta$, we obtain
$$\I_{a,b}'(\gamma)= a B \int_{0}^{\pi/2}  \cos\theta \left[(A+B \sin \theta)^{a-1}-(A-B \sin \theta)^{a-1} \right] \, \left[ |\sin  (\theta-\gamma)|^b - |\sin  (\theta+\gamma)|^b \right] d\theta.$$

Clearly, if $a=0$ or $a=1$, then $\I_{z,b}'=0$ and the function $\I_{a,b}$ is constant and  $$\I_{a,b}'(0)=\I_{a,b}'(\frac{\pi}{2})=0.$$
For $\gamma,\theta\in (0,\frac{\pi}{2})$ and $b>0$, then 
$$|\sin  (\theta-\gamma)|^b - |\sin  (\theta+\gamma)|^b <0. $$
Indeed,  $\sin (\theta-\gamma)-\sin(\theta+\gamma)= -2 \cos (\theta) \sin (\gamma)<0$.\\
Therefore, 

\begin{enumerate}
    \item if $a\in(0,1)$ and $\gamma \in (0,\frac{\pi}{2})$, then $\I_{a,b}'(\gamma)>0$ and the mapping $\I_{a,b}$ is strictly increasing on $[0,\pi/2]$;
    \item if $a\in(1,\infty)$ and $\gamma \in (0,\frac{\pi}{2})$, then $\I_{a,b}'(\gamma)<0$ and the mapping $\I_{a,b}$ is strictly decreasing on $[0,\frac{\pi}{2}]$.
\end{enumerate}
\end{proof}

As a consequence, we get 
\begin{cor}
Let $r,\rho \in (0,1)$, and $q\geq 1$.  Then the mapping $ \gamma \mapsto \J_q(r, \rho;\gamma)$  on $[0, \frac{\pi}{2}]$ is
\begin{enumerate}
    \item constant for $q=1$ or $q=\frac{n}{n-1}$;
    \item strictly increasing   for $1<q<\frac{n}{n-1}$;
    \item strictly decreasing   for $q>\frac{n}{n-1}$.
\end{enumerate}
\end{cor}

\begin{cor}
Let $p\in (1,\infty]$ and $x\in \BB^n\setminus \{0\}$. Then 
\begin{enumerate}
    \item If $p=n$ or $p=\infty$, then $\gamma \mapsto \K_p(x;\ell_\gamma)$ is constant.
    \item If $p \in (1,n)$, then $$\max_{\gamma\in[0,\pi/2]} \K_p(x;\ell_\gamma)=\K_p(x;\ell_0)=\K_p(x;n_x).$$
    \item If $p \in (n,\infty)$, then $$\max_{\gamma\in[0,\pi/2]} \K_p(x;\ell_\gamma)=\K_p(x;\ell_{\pi/2})=\K_p(x;t_x).$$
\end{enumerate}
\end{cor}

Thus we obtain our main theorem.

\section{Computation of $\C_p(x)$}

We will start with two particular cases $p=n$ or $p=\infty$.
\subsection{Case $p=\infty$}\hfill

In this case, the mapping  $\gamma \mapsto \K_\infty(x,\ell_\gamma)$ is constant, hence $\K_\infty(x,\ell_\gamma)=\K_\infty(x,\ell_0)$, and 
\begin{eqnarray}
\K_\infty(x,\ell_0)&=&\int_{\Sp} |\eta_1|\, d\sigma(\eta)\nonumber\\
 &=& \frac{n-1}{n} \frac{V(\BB^{n-1})}{V(\BB^n)}\int_{-1}^{1} (1-t^2)^{\frac{n-3}{2}} |t|\, dt\nonumber\\   
  &=&  \frac{2(n-1)}{n} \frac{V(\BB^{n-1})}{V(\BB^n)}\int_{0}^{1} (1-t^2)^{\frac{n-3}{2}} t\, dt.   
\end{eqnarray}
Using the substitution $u=1-t^2$, we deduce
\begin{eqnarray}
\K_\infty(x,\ell_0)&=&\frac{(n-1)}{n} \frac{V(\BB^{n-1})}{V(\BB^n)}\int_{0}^{1} u^{\frac{n-3}{2}} \, du \nonumber\\
&=& \frac{2}{n} \frac{V(\BB^{n-1})}{V(\BB^n)}.
\end{eqnarray}
 By (\ref{CK}), we have
 \begin{eqnarray}
\C_\infty(x;\ell) &=& \frac{2(n-1)}{1-|x|^2} \K_\infty(x;\ell)\nonumber\\
 &=& \frac{4(n-1)}{n} \frac{V(\BB^{n-1})}{V(\BB^n)} \frac{1}{1-|x|^2}\nonumber\\
 &=& \frac{2(n-1) \Gamma(\frac{n}{2})}{\sqrt{\pi}\Gamma({\frac{n}{2}+1})}\frac{1}{1-|x|^2}.
 \end{eqnarray}

Hence, if $u=\mathcal{P}_h[\phi]$, where $\phi\in L^\infty(\Sp,\R)$, then 

\begin{equation}\label{3.4.}
   |\nabla u(x)| \leq \frac{2(n-1) \Gamma(\frac{n}{2})}{\sqrt{\pi}\Gamma({\frac{n}{2}+1})}\frac{1}{1-|x|^2} \|\phi\|_\infty. \\
\end{equation}
 
We should mention that the sharp inequality (\ref{3.4.}) can  also be  obtained  as follows. According to \cite[Corollary 1.2]{kmm}, see also \cite{Burgeth}, if $u$ is a bounded  hyperbolic harmonic function, with $|u|<1$, then
\begin{equation}\label{3.5.}
     |\nabla u(0)| \leq \frac{4(n-1)}{n} \frac{V(\BB^{n-1})}{V(\BB^n)} 
\end{equation}
 
Let $x\in \BB^n$ and  $\varphi_x$ be the M\"obius transformation such that $\varphi_x(0)=x$, see \cite[p. 7 (2.1.4)]{sto2016}. By the M\"obius invariance of $\Delta_h$, the function $u\circ \varphi_x$ is also bounded hyperbolic harmonic function with $\nabla( u\circ \varphi_x)(0)=-(1-|x|^2)\nabla u(x)$. Hence (\ref{3.4.}) follows by considering $u\circ \varphi_x$ in (\ref{3.5.}).
 
 \subsection{Case $p=n$}\hfill
 
 The conjugate of $n$ is $q= \frac{n}{n-1}$.
 \begin{eqnarray}
\J_q(r,\rho;\gamma) &=& \int_{-\pi}^{\pi} \left( 1+\rho^2-2\rho r \cos \theta \right)   \left|   \cos(\theta- \gamma)\right|^q d\theta. \nonumber\\
&=& (1+\rho^2) \int_{-\pi}^{\pi}|\cos \theta|^q \, d\theta\nonumber\\
&=& 4 (1+\rho^2) \int_{0}^{\pi/2}\cos^q \theta \, d\theta.
\end{eqnarray}

Using 
\begin{eqnarray}
\K_n(x;\ell_\gamma) &=&\frac{n-2}{2\pi} \int_0^1 (1-r^2)^{\frac{n-4}{2}}  r^{q+1}    \J_q(r,|x|;\gamma)\, dr\nonumber\\
&=& \frac{n-2}{2\pi} \left( \int_0^1 (1-r^2)^{\frac{n-4}{2}}  r^{q+1} \,dr \right) \left(4 (1+|x|^2) \int_{0}^{\pi/2}\cos^q \theta \, d\theta \right)\nonumber\\
&=& \frac{n-2}{\pi} \left(\int_0^1 (1-r)^{\frac{n-4}{2}} r^{\frac{q}{2}}\, dr \right) \left(  \int_{0}^{\pi/2}\cos^q \theta \, d\theta \right) (1+|x|^2).
\end{eqnarray}

Recall some properties of the beta function. Let $a,b>0$ 
\begin{equation}\label{beta:2.6}
    \B(a,b)=\int_0^1 t^{a-1} (1-t)^{b-1} \, dt=2\int_0^{\pi/2} (\sin t)^{2a-1} (\cos t)^{2b-1} \, dt =\frac{\Gamma(a)\Gamma(b)}{\Gamma(a+b)}.  
\end{equation}
Therefore,
\begin{eqnarray}
  \K_n(x;\ell_\gamma)&=& \frac{n-2}{2\pi} \B(\frac{n}{2}-1,\frac{q}{2}+1)\B(\frac{1+q}{2},\frac{1}{2}) (1+|x|^2)\nonumber\\
  &=& \frac{n-2}{2\sqrt{\pi}}\frac{\Gamma(\frac{n}{2}-1)\Gamma(\frac{q+1}{2})}{\Gamma(\frac{n+q}{2})} (1+|x|^2).
 \end{eqnarray}

\subsection{Case $1<p<n$}
The following lemma is useful to compute  $\K_p(x;\ell_0)$.
\begin{lem}\label{lemma5} For $a>-1$ and $b>-1$
\begin{enumerate}
    \item $\int_{-1}^1 |t|^a(1-t^2)^b dt=\B(\frac{a+1}{2},b+1).$
\item  If $|u|<1$ and $\alpha \in \R$, then 
$$\int_{-1}^1 (1-ut)^{-\alpha} |t|^a(1-t^2)^b  dt=\B(\frac{a+1}{2},b+1)\,\, {}_3F_2(\frac{\alpha}{2},\frac{\alpha+1}{2},\frac{a+1}{2};\frac{1}{2},\frac{a+3}{2}+b;u^2).
$$
\end{enumerate}
\end{lem}

\begin{proof}
(1) Using the substitution $t=\sin \theta$ and (\ref{beta:2.6}), we get
\begin{eqnarray}
\int_{-1}^1|t|^a(1-t^2)^b dt&=&2\int_{0}^1t^a(1-t^2)^b dt\nonumber\\
&=&2\int_0^{\pi/2}(\sin \theta)^a(\cos \theta)^{2b+1}d\theta\nonumber\\
&=& \B(\frac{a+1}{2},b+1).\nonumber
\end{eqnarray}
 (2) Since $|ut|\le |u|<1$, for $|t| \leq 1$, we have 
\[(1-ut)^{-\alpha}=\sum_{k=0}^\infty \frac{(\alpha)_k}{k!}u^kt^k,\] and this series converges uniformly in $[-1,1]$, thus 
\begin{eqnarray}
\int_{-1}^1 (1-ut)^{-\alpha}|t|^a(1-t^2)^b\, dt&=&
\sum_{k=0}^{+\infty}
\int_{-1}^1\frac{(\alpha)_k}{k!}u^kt^k|t|^a(1-t^2)^b \, dt\nonumber\\
&=&\sum_{k=0}^{+\infty}
\int_{-1}^1\frac{(\alpha)_{2k}}{(2k)!} u^{2k}t^{2k}|t|^a(1-t^2)^b\, dt\nonumber\\
&=&\sum_{k=0}^{+\infty}
\frac{(\alpha)_{2k}}{(2k)!}\left(\int_{-1}^1t^{2k}|t|^a(1-t^2)^b\, dt\right)u^{2k}\label{2.9}.
\end{eqnarray}

Using Lemma \ref{lemma5} (1), we have 
$\int_{-1}^1t^{2k}|t|^a (1-t^2)^b dt =\frac{\Gamma(k+\frac{a+1}{2})\Gamma(b+1)}{\Gamma(k+\frac{a+3}{2}+b)}.$\\
Thus
$$\sum_{k=0}^{+\infty}
\frac{(\alpha)_{2k}}{(2k)!}\left(\int_{-1}^1t^{2k}|t|^a(1-t^2)^b dt\right)u^{2k}=\sum_{k=0}^{+\infty}
\frac{(\alpha)_{2k}}{(2k)!}\frac{\Gamma(k+\frac{a+1}{2})\Gamma(b+1)}{\Gamma(k+\frac{a+3}{2}+b)}u^{2k}.$$

Using 
$$(\alpha)_{2k}=2^{2k}(\frac{\alpha}{2})_k(\frac{\alpha+1}{2})_k,$$ 
$$(2k)!=\Gamma(2(k+1/2))=\frac{2^{2k}\Gamma(k+1/2)k!}{\sqrt{\pi}}=2^{2k}(1/2)_k k!,$$
and 

$$(x)_k=\frac{\Gamma(x+k)}{\Gamma(x)}, \mbox{ for } x\in \R \setminus \Z_-,$$  we get 

\begin{equation}
  \frac{(\alpha)_{2k}}{(2k)!}=\frac{(\frac{\alpha}{2})_k  (\frac{\alpha+1}{2})_k}{(\frac{1}{2})_k\, k!}.  
\end{equation}

\noindent Therefore
\begin{eqnarray*}
\sum_{k=0}^{+\infty}
\frac{(\alpha)_{2k}}{(2k)!}\left(\int_{-1}^1t^{2k}|t|^a(1-t^2)^b dt\right)u^{2k}&=&\sum_{k=0}^{+\infty}
\frac{(\frac{\alpha}{2})_k(\frac{\alpha+1}{2})_k}{(1/2)_k k!}\frac{\Gamma(k+\frac{a+1}{2})\Gamma(b+1)}{\Gamma(k+\frac{a+3}{2}+b)}u^{2k}\\
&=& \frac{\Gamma(\frac{a+1}{2})\Gamma(b+1)}
{\Gamma(\frac{a+3}{2}+b)}\sum_{k=0}^{+\infty}
\frac{(\frac{\alpha}{2})_k(\frac{\alpha+1}{2})_k(\frac{a+1}{2})_k}{(1/2)_k(\frac{a+3}{2}+b)_kk!}u^{2k}\\
&=&\B(\frac{a+1}{2},b+1) \,\, {}_3F_2(\frac{\alpha}{2},\frac{\alpha+1}{2},\frac{a+1}{2};\frac{1}{2},\frac{a+3}{2}+b;u^2).
\end{eqnarray*}
\end{proof}

Recall that 
\[\K_p(x;\ell_0)= \int_\Sp |\eta-x|^{2(n-1)(q-1)}\left|\langle\eta,\ell\rangle\right|^q\, d \sigma(\eta),\]
with  $\ell_0=e_1=\frac{x}{|x|};\, x=|x|e_1$.\\
We have $|\eta-x|^2=1+|x|^2-2|x|\eta_1$ and $\langle\eta,\ell_0\rangle=|\eta_1|.$\\
Let $\alpha=(n-1)(1-q)$ and $u=\frac{2|x|}{1+|x|^2}$.
$$\begin{array}{ccl}
    \K_p(x;\ell_0)&=&\int_\Sp (1+|x|^2-2|x|\eta_1)^{-\alpha}|\eta_1|^q\, d \sigma(\eta)\\
     &=&(1+|x|^2)^{-\alpha}\int_\Sp\left(1-2\frac{|x|}{1+|x|^2}\eta_1\right)^{-\alpha}|\eta_1|^q \, d \sigma(\eta)\\ &=& (1+|x|^2)^{-\alpha}\int_\Sp\left(1-u\eta_1\right)^{-\alpha}|\eta_1|^q \, d \sigma(\eta).\end{array}$$
As we integrate a function which depends on one variable on $\Sp$, then by the slice integration on spheres, we have 
$$\int_\Sp\left(1-u\eta_1\right)^{-\alpha}|\eta_1|^q \, d \sigma(\eta)=\frac{n-1}{n}\frac{V(\BB^{n-1})}{V(\BB^n)}\int_{-1}^1 (1-ut)^{-\alpha}|t|^q(1-t^2)^\frac{n-3}{2}dt.$$

By Lemma \ref{lemma5} (2), it yields
\[\begin{array}{l}
\int_{-1}^1 (1-ut)^{(n-1)(q-1)}|t|^q(1-t^2)^\frac{n-3}{2}dt\\= \frac{\Gamma(\frac{q+1}{2})\Gamma(\frac{n-1}{2})}
{\Gamma(\frac{q+n}{2})}\, {}_3F_2\left(\frac{(n-1)(1-q)}{2},\frac{(n-1)(1-q)+1}{2},\frac{q+1}{2};\frac{1}{2},\frac{q+n}{2};\frac{4|x|^2}{(1+|x|^2)^2}\right).
\end{array}\]

Therefore
\[\begin{array}{l}
      \int_\Sp\left(1-u\eta_1\right)^{(n-1)(q-1)}|\eta_1|^q \, d \sigma(\eta)  \\ =\frac{n-1}{n}\frac{V(\BB^{n-1})}{V(\BB^n)}\frac{\Gamma(\frac{q+1}{2})\Gamma(\frac{n-1}{2})}
{\Gamma(\frac{q+n}{2})}\, {}_3F_2\left(\frac{(n-1)(1-q)}{2},\frac{(n-1)(1-q)+1}{2},\frac{q+1}{2};\frac{1}{2},\frac{q+n}{2};\frac{4|x|^2}{(1+|x|^2)^2}\right)\\
=\frac{2}{n\sqrt{\pi}}\frac{\Gamma(\frac{q+1}{2})\Gamma(\frac{n}{2} +1)}
{\Gamma(\frac{q+n}{2})}\, {}_3F_2\left(\frac{(n-1)(1-q)}{2},\frac{(n-1)(1-q)+1}{2},\frac{q+1}{2};\frac{1}{2},\frac{q+n}{2};\frac{4|x|^2}{(1+|x|^2)^2}\right)\\
=\frac{1}{\sqrt{\pi}}\frac{\Gamma(\frac{q+1}{2})\Gamma(\frac{n}{2} )}
{\Gamma(\frac{q+n}{2})}\, {}_3F_2\left(\frac{(n-1)(1-q)}{2},\frac{(n-1)(1-q)+1}{2},\frac{q+1}{2};\frac{1}{2},\frac{q+n}{2};\frac{4|x|^2}{(1+|x|^2)^2}\right).\end{array}\]

\noindent Finally,
\[\begin{array}{l}
  \K_p(x;\ell_0)\\=(1+|x|^2)^{(n-1)(q-1)} \frac{\Gamma(\frac{q+1}{2})\Gamma(\frac{n}{2} )}
{\Gamma(\frac{q+n}{2}) \sqrt{\pi}}\, {}_3F_2\left(\frac{(n-1)(1-q)}{2},\frac{(n-1)(1-q)+1}{2},\frac{q+1}{2};\frac{1}{2},\frac{q+n}{2};\frac{4|x|^2}{(1+|x|^2)^2}\right).\end{array}\]

\subsection{Case $p>n$}
\noindent First, we will need the following lemma

\begin{lem}
Let $p\in \IN$  and $q$ be a positive real number and $n \geq 3$. Then
\begin{eqnarray}
    &(1)&\int_\Sp \eta_1^{2p+1} |\eta_2|^q\, d \sigma(\eta)=0.\\
    &(2)&\int_\Sp \eta_1^{2p} |\eta_2|^q d \sigma(\eta)=\frac{n-2}{\pi}\, \B(\frac{n}{2}-1,p+\frac{q}{2}+1)\, \B(p+\frac{1}{2},\frac{q+1}{2}).
\end{eqnarray}    
\end{lem}

\begin{proof}
Using (\ref{polar}) we have
\begin{eqnarray}
 \int_\Sp \eta_1^k |\eta_2|^q d \sigma(\eta)=\frac{n-2}{2\pi}\left(  \int_0^1 (1-r^2)^{\frac{n-4}{2}}r^{k+q+1}\, dr  \right)\left(\int_0^{2\pi}  \cos^k \theta |\sin \theta |^q \, d \theta \right).\nonumber
    \end{eqnarray}
    One can check that $\int_0^{2\pi}\cos^k \theta |\sin \theta |^q \, d \theta=0$ if $k$ is odd and if $k=2p$, then by (\ref{beta:2.6}), we get 
     \begin{eqnarray}
      \int_\Sp \eta_1^{2p} |\eta_2|^q d \sigma(\eta) &=& \frac{n-2}{\pi} \B(\frac{n}{2}-1,p+\frac{q}{2}+1) \B(p+\frac{1}{2},\frac{q+1}{2})\nonumber\\
      &=& \frac{n-2}{\pi}\,\Gamma(\frac{n-2}{2})\, \Gamma(\frac{q}{2}+1) \frac{\Gamma(p+\frac{1}{2})}{\Gamma(p+\frac{q+n}{2})}.\nonumber
     \end{eqnarray}
\end{proof}
\begin{rem}
For $n=2$ we obtain 
$$\int_{\mathbb{S}^1} \eta_1^{2p} |\eta_2|^q d \sigma(\eta)=  \B(p+\frac{1}{2},\frac{q+1}{2}).$$\\
\end{rem}
\noindent In the case $p>n$, the maximum of $\K_p(x,\ell_\gamma)$ is reached in the tangential direction, i.e., $\gamma=\frac{\pi}{2}$ and 
\begin{eqnarray}
  \K_p(x;\ell_{\pi/2})= \int_{\Sp} \left( 1+|x|^2-2|x|\eta_1 \right)^{(n-1)(q-1)} \left| \eta_2  \right|^q \, d \sigma(\eta).
\end{eqnarray}
To simplify the notation, we consider 
\begin{equation}
    u=\frac{2|x|}{1+|x|^2} \mbox{ and } \alpha=(n-1)(1-q).
\end{equation}
Thus we obtain
\begin{eqnarray}
    \K_p(x;\ell_{\pi/2})&=& (1+|x|^2)^{-\alpha}\int_{\Sp} \left( 1-u\eta_1 \right)^{-\alpha} \left| \eta_2  \right|^q \, d \sigma(\eta)\nonumber\\ 
    &=& (1+|x|^2)^{-\alpha} \sum_{k=0}^\infty \frac{(\alpha)_k\,u^k}{k!} \int_\Sp \eta_1^k |\eta_2|^q \sigma(\eta) \nonumber\\
    &=& \frac{n-2}{\pi} \Gamma(\frac{n-2}{2})\, \Gamma(\frac{q+1}{2})  (1+|x|^2)^{-\alpha} \sum_{k=0}^\infty  \frac{(\alpha)_{2k}}{(2k)!} \frac{\Gamma(k+\frac{1}{2})}{\Gamma(k+\frac{q+n}{2}) } u^{2k}.\nonumber
\end{eqnarray}
Since 
\begin{eqnarray}
   \sum_{k=0}^\infty  \frac{(\alpha)_{2k}}{(2k)!} \frac{\Gamma(k+\frac{1}{2})}{\Gamma(k+\frac{q+n}{2})}u^{2k} &=& \frac{\Gamma(\frac{1}{2}) }{\Gamma(\frac{q+n}{2})} \sum_{k=0}^\infty \frac{(\frac{\alpha}{2})_k (\frac{\alpha+1}{2})_k}{(\frac{q+n}{2})_k k!} u^{2k}\nonumber \\
   &=& \frac{\Gamma(\frac{1}{2}) }{\Gamma(\frac{q+n}{2})} \, {}_2F_1(\frac{\alpha}{2},\frac{\alpha+1}{2} ; \frac{q+n}{2};u^2).
\end{eqnarray}

\noindent Therefore,
\begin{eqnarray*}
    \K_p(x;\ell_{\pi/2})&=& \frac{n-2}{\sqrt{\pi}} \frac{\Gamma(\frac{n-2}{2}) \Gamma(\frac{q+1}{2})}{\Gamma(\frac{q+n}{2})}  (1+|x|^2)^{-\alpha} \, {}_2F_1(\frac{\alpha}{2},\frac{\alpha+1}{2}; \frac{q+n}{2};u^2).
     \end{eqnarray*}

Using the following well-known transformation formula due to Kummer

$$
{}_2F_1\left(a, a+\frac{1}{2};c; \frac{4v}{(1+v)^2} \right)= (1+v)^{2a}\, {}_2F_1\left(2a, 2a-c+1;c; v \right), 
$$
which is respectively a slight variation of the one  given in \cite[Section 15.3 (20) ]{AS}.\\

\noindent It yields 
\begin{eqnarray}
   \K_p(x;\ell_{\pi/2})&=&
\frac{n-2}{\sqrt{\pi}} \frac{\Gamma(\frac{n-2}{2}) \Gamma(\frac{q+1}{2})}{\Gamma(\frac{q+n}{2})}\, {}_2F_1\left(\alpha, \alpha-\frac{q+n}{2}+1;\frac{q+n}{2}; |x|^2 \right)\\
&=& \frac{n-2}{\sqrt{\pi}} \frac{\Gamma(\frac{n-2}{2})\, \Gamma(\frac{q+1}{2})}{\Gamma(\frac{q+n}{2})}\, {}_2F_1\left((n-1)(1-q), \frac{n}{2}+q(\frac{1}{2}-n);\frac{q+n}{2}; |x|^2 \right).\nonumber 
\end{eqnarray}

\end{document}